\documentclass[a4paper, 12pt]{amsart}
\usepackage{amsmath, amssymb, amsthm, graphicx, stmaryrd}
\usepackage{mathrsfs}
\usepackage{comment}
\usepackage[all]{xy}
\usepackage{tikz}
\usetikzlibrary{decorations.markings}

\usepackage[top=40mm,bottom=40mm,left=30mm,right=30mm]{geometry}

\theoremstyle{definition}
\newtheorem{theorem}{Theorem}[section]

\newtheorem{lemma}[theorem]{Lemma}
\newtheorem{proposition}[theorem]{Proposition}

\makeatletter
\@addtoreset{equation}{section} 
\makeatother

\DeclareMathOperator{\cl}{cl}
\DeclareMathOperator{\scl}{scl}
\DeclareMathOperator{\ab}{ab}

\title[scl in braided Ptolemy-Thompson groups]{A note on stable commutator length in
braided Ptolemy-Thompson groups}
\author{Shuhei Maruyama}
\address{Graduate School of Mathematics,
Nagoya University, Japan}
\email{m17037h@math.nagoya-u.ac.jp}

\begin{document}

\begin{abstract}
  In this note, we show that the sets of
  all stable commutator lengths in
  the braided Ptolemy-Thompson groups are
  equal to non-negative rational numbers.
\end{abstract}

\maketitle

\section{Introduction}
Let $G$ be a group and $[G,G]$ its commutator subgroup.
For an element $g \in [G,G]$, its commutator length
$\cl(g)$ is defined by
$\cl(g) = \min \{ l \mid g = [a_1, b_1]\cdots [a_l, b_l] \}$.
For $g \in G$
with $g^k \in [G,G]$ for some positive integer $k$,
its stable commutator length $\scl(g)$ is defined
by $\scl(g) = \lim_{n \to \infty} \cl(g^{kn})/kn$.
For $g \in G$ with $g^k \notin [G,G]$ for any $k$,
we put $\scl(g) = \infty$.
We say that the group $G$ has a gap in stable commutator length if
there is a positive constant $c$ such that either $\scl(g) \geq c$
or $\scl(g) = 0$ holds for any $g \in G$.

The values of stable commutator length
in many classes of groups has been studied.
For example, all stable commutator lengths are rational
in the central extension of the Thompson group $T$ corresponding to
the Euler class(Ghys-Sergiescu\cite{ghysserg87}),
free groups(Calegari\cite{calegari09}),
free products of some classes of groups
(Walker\cite{walker13}, Calegari\cite{calegari11}, Chen\cite{chen18a}),
amalgams of free abelian groups(Susse\cite{susse15}), and
certain graphs of groups
(Clay-Forester-Louwsma\cite{clay_forester_louwsma16},
Chen\cite{chen19}).
It is known that there is a finitely presented group which has an
irrational stable commutator length(Zhuang\cite{zhuang08}).
The gap in many classes of groups has also been studied.
For example, the above central extension of the Thompson group
has no gap, and the following groups have a gap;
free groups(Duncan-Howie\cite{duncan_howie91}),
word-hyperbolic groups(Calegari-Fujiwara\cite{calegari_fujiwara10}),
finite index subgroups of the mapping class group
of a closed orientable surface
(Bestvina-Bromberg-Fujiwara\cite{bestvina_bromberg_fujiwara16}),
right-angled Artin groups(Heuer\cite{heu19}),
and the fundamental group of closed oriented connected 3-manifolds
(Chen-Heuer\cite{chen_heuer19}).

In this note, we show that all stable commutator lengths in the
braided Ptolemy-Thompson groups $T^*$ and $T^{\sharp}$
(see Section \ref{section:braided_Thompson_groups})
are rational
and these groups have no gap.
More precisely, we show the following theorem.

\begin{theorem}\label{theorem:rational}
  The sets $\scl(T^*)$ and $\scl(T^{\sharp})$
  of all stable commutator lengths in the braided
  Ptolemy-Thompson
  groups are equal to the non-negative rational numbers
  $\mathbb{Q}_{\geq 0}$.
\end{theorem}

\section{Preliminaries}
\subsection{quasi-morphisms and Bavard's duality theorem}\label{subsection:quasi_Bavard}

Let $G$ be a group.
A map $\phi:G \to \mathbb{R}$ is a quasi-morphism if
\[
  D(\phi) = \sup_{g, h \in G}|\phi(gh) - \phi(g) - \phi(h)|
\]
is finite.
This $D(\phi)$ is called the defect of $\phi$.
A quasi-morphism is homogeneous if the condition
$\phi(g^n) = n \phi(g)$ holds for any $g \in G$ and
$n \in \mathbb{Z}$.
Let $Q(G)$ denote the $\mathbb{R}$-vector space consisting
of all homogeneous quasi-morphisms on $G$.
The key ingredient to calculate the stable commutator length
is the following Bavard's duality theorem.
\begin{theorem}\cite{bavard91}
  Let $G$ be a group and $g \in [G,G]$, then
  \[
    \scl(g) = \sup_{\phi \in Q(G)}
    \frac{|\phi(g)|}{2D(\phi)}.
  \]
\end{theorem}
This theorem allows us to calculate the stable commutator length
in groups with few homogeneous quasi-morphisms.
For example, any homogeneous quasi-morphism on
the infinite braid group $B_{\infty}$ is equal to
the abelianization homomorphism $B_{\infty} \to \mathbb{Z}$
up to constant multiple(see Kotschick\cite{kotschick08}).
Thus, we have $\scl(g) = 0$ for
$g \in [B_{\infty}, B_{\infty}]$.
The central extension $\widetilde{T}$
of the Thompson group $T$ corresponding
to the Euler class is also a group with few quasi-morphisms.
It is known that any homogeneous quasi-morphism on
$\widetilde{T}$ is equal to the rotation number up to constant
multiple.
Ghys-Sergiescu\cite{ghysserg87} showed that the
set of all rotation numbers on $\widetilde{T}$
coincides with the rational numbers $\mathbb{Q}$.
Since $\widetilde{T} = [\widetilde{T}, \widetilde{T}]$
and the defect of rotation number is equal to $1$,
we have $\scl(\widetilde{T}) = \mathbb{Q}_{\geq 0}$.

\subsection{Central extensions of the Thompson group $T$}
Let $n$ be an integer
and $0 \to \mathbb{Z} \to T_n \xrightarrow{p} T \to 1$
denote the central $\mathbb{Z}$-extension of $T$
corresponding to $n$ times the Euler class
$e \in H^2(T;\mathbb{Z})$.
Let $\chi$ be a bounded cocycle representing the Euler class
$e \in H^2(T;\mathbb{Z})$.
Then the extension $T_n$ can be obtained as a group
$T \times \mathbb{Z}$ with the product
\[
  (s, i)\dot (t, j) = (st, i + j + n\chi(s, t)).
\]
Since the Thompson group $T$ is uniformly perfect and
the Euler class $e$ is bounded,
there is a unique homogeneous quasi-morphism
$\phi_n \in Q(T_n)$ satisfying
$-[\delta \phi_n] = n p^* e_b \in H_b^2(T_n;\mathbb{R})$
(see Barge-Ghys\cite{bargeghys92}).
Here $e_b \in H_b^2(T;\mathbb{R})$ is the bounded Euler class.
The homogeneous quasi-morphism $\phi_n$ is obtained as
the homogenization of the quasi-morphism
$T_n \to \mathbb{R}; (t,j) \mapsto j$.
More explicitly, the quasi-morphism $\phi_n$ is written as
\[
  \phi_n(t, j) = j + n\lim_{k\to \infty} \frac{a(t, k)}{k},
\]
where
$a(t, k) =
\chi(t, t) + \chi(t^2, t) + \cdots + \chi(t^{k-1}, t)
\in \mathbb{Z}$.
Since $T_1$ is equal to $\widetilde{T}$,
the homogeneous quasi-morphism $\phi_1 : T_1 \to \mathbb{R}$
coincides with the rotation number on $\widetilde{T}$.
It is known that the defect of the rotation number
on $\widetilde{T}$
is equal to $1$(see, for example, Calegari\cite{calegari09}),
thus we have $D(\phi_1) = 1$.

\begin{lemma}\label{lemma:rotation_defect}
  Let $n$ be a non-zero integer, then $D(\phi_n) = n$.
\end{lemma}

\begin{proof}
  Note that
  \begin{align*}
    1 &= D(\phi_1) = \sup_{(s, i), (t, j)}
    |\phi_1(st, i + j + \chi(s, t)) -
    \phi_1(s, i) - \phi_1(t, j)|\\
    &= \sup_{s,t} \left| \chi(s,t) + \lim_{k \to \infty}\frac{
    a(st, k) - a(s, k) - a(s, k)}{k} \right|.
  \end{align*}
  Thus we have
  \begin{align*}
    D(\phi_n) &= \sup_{(s, i), (t, j)}
    |\phi_n(st, i + j + \chi(s, t)) -
    \phi_n(s, i) - \phi_n(t, j)|\\
    &= \sup_{s,t} \left| n\chi(s,t) + n\lim_{k \to \infty}\frac{
    a(st, k) - a(s, k) - a(s, k)}{k} \right|\\
    & =nD(\phi_1)= n.
  \end{align*}
\end{proof}

\begin{lemma}\label{lemma:rotation_rational}
  Let $n$ be a non-zero integer, then
  $\phi_n(T_n) = \mathbb{Q}$.
\end{lemma}

\begin{proof}
  Ghys-Sergiescu\cite{ghysserg87} showed that
  $\phi_1(T_1)$ is equal to the rational numbers $\mathbb{Q}$.
  Thus, for any $(t,j) \in T_1$, the number
  $\lim(a(t, k))/k = \phi_1(t, j) - j$ is rational and
  therefore $\phi_n(T_n) \subset \mathbb{Q}$ follows.
  Take $q \in \mathbb{Q}$.
  Then,
  for any $q \in \mathbb{Q}$, there is an element
  $(t, j) \in T_1$ such that
  \[
    \frac{q}{n} = \phi_1(t, j)
    = j + \lim_{k\to \infty} \frac{a(t, k)}{k}.
  \]
  Thus we have
  $q = \phi_n(t, nj)$.
\end{proof}
Since $\dim Q(T_n) = 1$ and $T_n = [T_n, T_n]$,
we have $\scl(T_n) = \mathbb{Q}_{\geq 0}$ for
any non-zero integer $n$
.

\subsection{The braided Ptolemy-Thompson groups}\label{section:braided_Thompson_groups}
In \cite{funarkapou08}, Funar and Kapoudjian introduced
two extensions $T^{*}$ and $T^{\sharp}$
of Thompson group $T$
by the infinite braid group $B_{\infty}$
and showed that they are finitely presented.
Moreover, it is shown that their abelianization
$H_1(T^*)$ and $H_1(T^{\sharp})$ are isomorphic to
$\mathbb{Z}/12\mathbb{Z}$ and $\mathbb{Z}/6\mathbb{Z}$
respectively, and therefore $T^*$ and $T^{\sharp}$ are
not isomorphic.
Let $\diamond$ denote either $*$ or $\sharp$.
Since the abelianization homomorphism
$B_{\infty} \to \mathbb{Z}$ is
$T^{\diamond}$-conjugation invariant,
we have central $\mathbb{Z}$-extensions
$T_{\ab}^{\diamond}$ of $T$, that is,
we have the diagram
\begin{equation}\label{diagram}
  \xymatrix{
  1 \ar[r] & B_{\infty} \ar[r] \ar[d]
  & T^{\diamond} \ar[r] \ar[d]^{\rho} & T \ar[r] \ar@{=}[d] & 1\\
  0 \ar[r] & \mathbb{Z} \ar[r] & T_{\ab}^{\diamond} \ar[r]^{p}
  & T \ar[r] & 1.
  }
\end{equation}
The corresponding cohomology classes $e(T_{\ab}^{*})$ and
$e(T_{\ab}^{\sharp})$ in $H^2(T;\mathbb{Z})$ are equal to
$12 e$ and $21 e$ respectively, where
$e \in H^2(T;\mathbb{Z})$ is the Euler class
(see Funar-Sergiescu\cite{funarserg10} and
Funar-Kapoudjian-Sergiescu\cite{funarkapouserg12}).
Thus $T_{\ab}^*$ is isomorphic to $T_{12}$ and
$T_{\ab}^{\sharp}$ is isomorphic to $T_{21}$.

\section{The stable commutator length in the braided Ptolemy-Thompson groups}
In this section, we show that the stable commutator lengths
in $T^{\diamond}$ are rational and all positive rationals are
realized as the stable commutator length.

At first, we show the following proposition,
which is essentially proved in Shtern\cite{shtern94}.
\begin{proposition}\label{prop:quasimorp_exact}
  For an exact sequence
  $1 \to K \to G \xrightarrow{p} H \to 1$,
  we have an exact sequence
  \[
    0 \to Q(H) \to Q(G) \to Q(K)^G,
  \]
  where $Q(K)^G$ is the vector space consisting of all
  $G$-conjugation invariant homogeneous quasi-morphisms
  on $K$.
\end{proposition}

\begin{proof}
  Shtern\cite{shtern94} showed that, if $\phi \in Q(G)$
  is equal to $0$ on $K$, there exists
  a homogeneous quasi-morphism $\psi \in Q(H)$ such that
  $\phi = p^*\psi$.
  This implies that the above sequence is exact at
  $Q(G)$.
\end{proof}

Let us consider the diagram(\ref{diagram}).
There are homogeneous quasi-morphisms
$\phi_* = \phi_{12} \in Q(T_{12}) = Q(T_{\ab}^*)$ and
$\phi_{\sharp} = \phi_{21} \in Q(T_{21}) = Q(T_{\ab}^{\sharp})$.
By Proposition \ref{prop:quasimorp_exact} and
the exact sequence $1 \to [B_{\infty}, B_{\infty}] \to
T^{\diamond} \xrightarrow{\rho} T_{\ab}^{\diamond} \to 1$,
we have an injection
$\rho^*:Q(T_{\ab}^{\diamond}) \to Q(T^{\diamond})$.
Thus we have a non-trivial homogeneous quasi-morphism
$\psi_{\diamond} = \rho^*\phi_{\diamond}
\in Q(T^{\diamond})$ satisfying
$[\delta \psi_{\diamond}] = (p \circ \rho)^* e_b \in
H_b^2(T^{\diamond};\mathbb{R})$.

\begin{lemma}\label{lemma:dim_bPT=1}
  The dimension of $Q(T^{\diamond})$ is equal to $1$.
\end{lemma}

\begin{proof}
  Since there is a non-trivial element
  $\psi \in Q(T^{\diamond})$, it is enough to show that
  the inequality $\dim Q(T^{\diamond}) \leq 1$ holds.
  For an exact sequence
  $1 \to B_{\infty} \to T^{\diamond} \to T \to 1$,
  we have an exact sequence
  \[
    0 \to Q(T) \to Q(T^{\diamond}) \to
    Q(B_{\infty})^{T^{\diamond}}
  \]
  by Proposition \ref{prop:quasimorp_exact}.
  Since the Thompson group $T$ is uniformly perfect,
  we have $Q(T) = 0$.
  Thus the map
  $Q(T^{\diamond}) \to Q(B_{\infty})^{T^{\diamond}}$
  is injective.
  The $\mathbb{R}$-vector space
  $Q(B_{\infty})$ is spanned by the
  abelianization homomorphism $B_{\infty} \to \mathbb{Z}$
  (see Kotschick\cite{kotschick08}).
  Moreover, since the abelianization homomorphism is
  $T^{\diamond}$-conjugation invariant,
  we have $Q(B_{\infty}) = Q(B_{\infty})^{T^{\diamond}}$.
  Thus we have
  $\dim Q(T^{\diamond}) \leq \dim Q(B_{\infty})^{T^{\diamond}}
  = \dim Q(B_{\infty}) = 1$.
\end{proof}

\begin{proof}[Proof of Theorem \ref{theorem:rational}]
  Note that, by Lemma \ref{lemma:rotation_defect} and
  the surjectivity of $\rho: T^{\diamond} \to T_{\ab}^{\diamond}$,
  the defect $D(\psi_*) = D(\rho^* \phi_*)$
  is equal to $D(\phi_*) = 12$
  and $D(\psi_{\sharp}) = D(\rho^* \phi_{\sharp})$
  is equal to $D(\phi_{\sharp}) = 21$.
  Since $H_1(T^*) = \mathbb{Z}/12\mathbb{Z}$,
  the power $g^{12}$ is in $[T^*, T^*]$ for any $g \in T^*$.
  Thus, by Lemma \ref{lemma:dim_bPT=1} and
  the Bavard's duality theorem, we have
  \[
    \scl(g) = \frac{\scl(g^{12})}{12}
    = \frac{|\psi_*(g^{12})|}{288} =
    \frac{|\phi_*(\rho(g))|}{24}
  \]
  for any $g \in [T^*, T^*]$.
  Thus, by Lemma \ref{lemma:rotation_rational},
  we have $\scl(T^*) = \mathbb{Q}_{\geq 0}$.
  We can prove $\scl(T^{\sharp}) = \mathbb{Q}_{\geq 0}$
  in the same way.
\end{proof}


\begin{thebibliography}{99}

\bibitem{bargeghys92} J. Barge and \'{E}. Ghys:
Cocycles d'{E}uler et de Maslov,
Math. Ann. {\bf 294} (1992), 235--265.

\bibitem{bavard91} C. Bavard:
Longueur stable des commutateurs,
Enseign. Math. (2) {\bf 37} (1991), 109--150.

\bibitem{bestvina_bromberg_fujiwara16} M. Bestvina, K. Bromberg
and K. Fujiwara:
Stable commutator length on mapping class groups,
Ann. Inst. Fourier {\bf 66} (2016), 871--898.

\bibitem{calegariscl} D. Calegari:
scl, MSJ Memoirs {\bf 20},
Mathematical Society of Japan (2009).

\bibitem{calegari09} D. Calegari:
Stable commutator length is rational in free groups,
J. Amer. Math. Soc. {\bf 22} (2009), 941--961.

\bibitem{calegari11} D. Calegari:
scl, sails, and surgery,
J. Topol. {\bf 4} (2011), 305--326.

\bibitem{calegari_fujiwara10} D. Calegari and K. Fujiwara:
Stable commutator length in word-hyperbolic groups,
Groups Geom. Dyn. {\bf 4} (2010), 59--90.

\bibitem{chen18a} L. Chen:
Scl in free products,
Algebr. Geom. Topol. {\bf 18} (2018), 3279--3313.

\bibitem{chen18b} L. Chen:
Spectral gap of scl in free products,
Proc. Amer. Math. Soc. {\bf 146} (2018), 3143--3151.

\bibitem{chen19} L. Chen:
Scl in graphs of groups,
arXiv:1904.08360.

\bibitem{chen_heuer19} L. Chen and N. Heuer:
Spectral gap of scl in graphs of groups and $3$-manifolds,
arXiv:1910.14146.

\bibitem{clay_forester_louwsma16} M. Clay, M. Forester, and J. Louwsma:
Stable commutator length in Baumslag-Solitar groups and quasimorphisms for tree actions, Trans. Amer. Math. Soc. {\bf 368} (2016), 4751--4785.

\bibitem{duncan_howie91} A. J. Duncan and J. Howie:
The genus problem for one-relator products of locally indicable groups,
Math. Z. {\bf 208} (1991), 225--237.

\bibitem{funarkapou08} L. Funar and C. Kapoudjian:
The braided Ptolemy-Thompson group is finitely presented,
Geom. Topol. {\bf 12} (2008), 475--530.

\bibitem{funarkapouserg12} L. Funar, C. Kapoudjian, and
V. Sergiescu:
Asymptotically rigid mapping class groups and Thompson's
groups, Handbook of Teichm\"{u}ller theory,
IRMA Lect. Math. Theor. Phys. {\bf 17} (2012), 595--664.

\bibitem{funarserg10} L. Funar and V. Sergiescu:
Central extensions of the Ptolemy-Thompson group and
quantized Teichm\"{u}ller theory,
J. Topol. {\bf 3} (2010), 29--62.

\bibitem{ghysserg87} \'{E} Ghys and V. Sergiescu:
Sur un groupe remarquable de diff\'{e}omorphismes du cercle,
Comment. Math. Helv. {\bf 62} (1987), 185--239.

\bibitem{heu19} N. Heuer:
Gaps in SCL for amalgamated free products and RAAGs,
Geom. Funct. Anal. {\bf 29} (2019), 198--237.

\bibitem{kotschick08} D. Kotschick:
Stable length in stable groups,
Groups of diffeomorphisms, Adv. Stud. Pure Math. {\bf 52}
(2008), 401--413.

\bibitem{shtern94} A. Shtern:
Quasisymmetry. {I},
Russian J. Math. Phys. {\bf 2} (1994), 353--382.

\bibitem{susse15} T. Susse:
Stable commutator length in amalgamated free products,
J. Topol. Anal. {\bf 7} (2015), 693--717.

\bibitem{walker13} A. Walker:
Stable commutator length in free products of cyclic groups,
Exp. Math. {\bf 22} (2013), 282--298.

\bibitem{zhuang08} D. Zhuang:
Irrational stable commutator length in finitely presented groups,
J. Mod. Dyn. {\bf 2} (2008), 499--507.

\end{thebibliography}
\end{document}